\documentclass[english,american]{article}
\usepackage[latin9]{inputenc}
\usepackage{geometry}
\usepackage{color}
\usepackage{babel}
\usepackage{units}
\usepackage{amsmath}
\usepackage{amsthm}
\usepackage{amssymb}
\usepackage{stmaryrd}
\usepackage[all]{xy}
\usepackage[unicode=true,pdfusetitle,
 bookmarks=true,bookmarksnumbered=false,bookmarksopen=false,
 breaklinks=false,pdfborder={0 0 0},pdfborderstyle={},backref=page,colorlinks=true]
 {hyperref}

\makeatletter
\numberwithin{equation}{section}
\numberwithin{figure}{section}
\theoremstyle{plain}
\newtheorem{thm}{\protect\theoremname}[section]
\theoremstyle{plain}
\newtheorem{fact}[thm]{\protect\factname}
\theoremstyle{plain}
\newtheorem{conjecture}[thm]{\protect\conjecturename}
\theoremstyle{remark}
\newtheorem{rem}[thm]{\protect\remarkname}
\theoremstyle{definition}
\newtheorem{defn}[thm]{\protect\definitionname}
\theoremstyle{plain}
\newtheorem{lem}[thm]{\protect\lemmaname}
\theoremstyle{plain}
\newtheorem{cor}[thm]{\protect\corollaryname}

\date{}

\usepackage[leftcaption]{sidecap}

\sidecaptionvpos{figure}{c}

\newcommand{\FigBesBeg}[1][1.0]{%
 \let\MyFigure\figure
 \let\MyEndfigure\endfigure
 \renewenvironment{figure}[1]{\begin{SCfigure}[#1]##1}{\end{SCfigure}}}

\newcommand{\FigBesEnd}{%
 \let\figure\MyFigure
 \let\endfigure\MyEndfigure}

\makeatother

\addto\captionsamerican{\renewcommand{\conjecturename}{Conjecture}}
\addto\captionsamerican{\renewcommand{\corollaryname}{Corollary}}
\addto\captionsamerican{\renewcommand{\definitionname}{Definition}}
\addto\captionsamerican{\renewcommand{\factname}{Fact}}
\addto\captionsamerican{\renewcommand{\lemmaname}{Lemma}}
\addto\captionsamerican{\renewcommand{\remarkname}{Remark}}
\addto\captionsamerican{\renewcommand{\theoremname}{Theorem}}
\addto\captionsenglish{\renewcommand{\conjecturename}{Conjecture}}
\addto\captionsenglish{\renewcommand{\corollaryname}{Corollary}}
\addto\captionsenglish{\renewcommand{\definitionname}{Definition}}
\addto\captionsenglish{\renewcommand{\factname}{Fact}}
\addto\captionsenglish{\renewcommand{\lemmaname}{Lemma}}
\addto\captionsenglish{\renewcommand{\remarkname}{Remark}}
\addto\captionsenglish{\renewcommand{\theoremname}{Theorem}}
\providecommand{\conjecturename}{Conjecture}
\providecommand{\corollaryname}{Corollary}
\providecommand{\definitionname}{Definition}
\providecommand{\factname}{Fact}
\providecommand{\lemmaname}{Lemma}
\providecommand{\remarkname}{Remark}
\providecommand{\theoremname}{Theorem}

\begin{document}
\selectlanguage{english}%
\global\long\def\genus{\mathrm{genus}}
 \global\long\def\Hom{\mathrm{Hom}}
 \global\long\def\wedger{{\textstyle \bigvee^{r}S^{1}} }
\global\long\def\Stab{\mathrm{Stab}}
\global\long\def\trw{{\cal T}r_{w} }
\global\long\def\trwl{{\cal T}r_{w_{1},\ldots,w_{\ell}} }
 \global\long\def\tr{{\cal T}r }
 \global\long\def\cl{{\cal \mathrm{cl}} }
 \global\long\def\sql{{\cal \mathrm{sql}} }
 \global\long\def\wg{{\cal \mathrm{Wg}} }
\global\long\def\moeb{\mathrm{M\ddot{o}b} }
\global\long\def\F{\mathrm{\mathbf{F}} }
 \global\long\def\fs{{\cal FS} }
 \global\long\def\auteq{\stackrel{\Aut(\F)}{\sim} }
 \global\long\def\auteqr{\stackrel{\Aut(\F_{r})}{\sim} }
 \global\long\def\id{\mathrm{id}}
 \global\long\def\e{\varepsilon}
 \global\long\def\U{\mathcal{U}}
 \global\long\def\O{\mathcal{O}}
 \global\long\def\Aut{\mathrm{Aut}}
 \global\long\def\E{\mathbb{E}}
 \global\long\def\Q{\mathbb{\mathbb{\mathbf{Q}}}}
\global\long\def\wl{w_{1},\ldots,w_{\ell}}
 \global\long\def\ssn{{\cal S}^{1}\wr S_{N}}
\global\long\def\cmsn{C_{m}\wr S_{N}}
 \global\long\def\ctsn{C_{2}\wr S_{N}}
 \global\long\def\std{\mathrm{std}}
\global\long\def\Xcov{{\scriptscriptstyle \overset{\twoheadrightarrow}{X}}}
 \global\long\def\covers{\leq_{\Xcov}}
 \global\long\def\ae{{\cal AE}}
 \global\long\def\alg{\le_{\mathrm{alg}}}
 \global\long\def\rk{\mathrm{rk}}
\global\long\def\ff{\stackrel{*}{\le}}
\global\long\def\H{{\cal H}}
 \global\long\def\ecm{\chi_{m}}
 \global\long\def\eci{\chi_{\infty}}
 \global\long\def\ect{\chi_{2}}
 \global\long\def\Xcov{{\scriptscriptstyle \overset{\twoheadrightarrow}{X}}}
 \global\long\def\covers{\leq_{\Xcov}}
 \global\long\def\XCO#1#2{\left[#1,#2\right)_{\Xcov}}
 \global\long\def\XF#1{\XCO{#1}{\infty}}

\title{Surface Words are Determined by Word Measures on Groups}

\author{Michael Magee ~and~ Doron Puder}

\date{\today }
\maketitle
\begin{abstract}
Every word $w$ in a free group naturally induces a probability measure
on every compact group $G$. For example, if $w=\left[x,y\right]$
is the commutator word, a random element sampled by the $w$-measure
is given by the commutator $\left[g,h\right]$ of two independent,
Haar-random elements of $G$. Back in 1896, Frobenius showed that
if $G$ is a finite group and $\psi$ an irreducible character, then
the expected value of $\psi\left(\left[g,h\right]\right)$ is $\frac{1}{\psi\left(e\right)}$.
This is true for any compact group, and completely determines the
$\left[x,y\right]$-measure on these groups. An analogous result holds
with the commutator word replaced by any surface word. 

We prove a converse to this theorem: if $w$ induces the same measure
as $\left[x,y\right]$ on every compact group, then, up to an automorphism
of the free group, $w$ is equal to $\left[x,y\right]$. The same
holds when $\left[x,y\right]$ is replaced by any surface word.

The proof relies on the analysis of word measures on unitary groups
and on orthogonal groups, which appears in separate papers, and on
new analysis of word measures on generalized symmetric groups that
we develop here.
\end{abstract}
\tableofcontents{}
\selectlanguage{american}%

\section{Introduction}

\selectlanguage{english}%
Let $\F_{r}$ be the free group on $r$ generators $x_{1},\ldots,x_{r}$,
and let $G$ be any finite, or more generally, compact group. Every
word $w\in\F_{r}$ induces a map, called a \emph{word map}, 
\[
w\colon\underbrace{G\times\ldots\times G}_{r~\mathrm{times}}\to G,
\]
defined by substitutions. For example, if $w=x_{1}x_{3}x_{1}x_{3}^{-2}\in\F_{3}$,
then $w\left(g_{1},g_{2},g_{3}\right)=g_{1}g_{3}g_{1}g_{3}^{-2}$.
The push-forward via this word map of the Haar probability measure
(uniform measure in the finite case) on $G\times\ldots\times G$ is
called the \emph{$w$-measure} on $G$. Put differently, for each
$1\le i\le r$, substitute $x_{i}$ with an independent, Haar-distributed
random element of $G$, and evaluate the product defined by $w$ to
obtain a random element in $G$ sampled by the $w$-measure. We say
the resulting element is a \emph{$w$-random} element of $G$. 

\subsection*{Measures induced by surface words}

The study of word measures in groups has its seeds in the 1896 work
of Frobenius \cite{frobenius1896gruppencharaktere}. Let $\left[x,y\right]=xyx^{-1}y^{-1}$
be\footnote{\selectlanguage{american}%
Throughout this paper, the letters $x$, $y$, and also $x_{i}$ and
$y_{j}$ denote different generators in the same basis of $\F_{r}$.\selectlanguage{english}%
} the commutator word. Frobenius shows that the $\left[x,y\right]$-measure
on a finite group $G$ is given by 
\[
\frac{1}{\left|G\right|}\sum_{\psi\in\mathrm{Irr}\left(G\right)}\frac{1}{\psi\left(e\right)}\psi,
\]
where $\mathrm{Irr}\left(G\right)$ marks the set of irreducible characters
of $G$ and $e$ is the identity element of $G$. As word measures
on finite groups are class functions, this is equivalent to the fact
that for every $\psi\in\mathrm{Irr}\left(G\right)$, the expected
value of $\psi$ under the $\left[x,y\right]$-measure is $\frac{1}{\psi\left(e\right)}=\frac{1}{\dim\psi}$.
In 1906, Frobenius and Schur \cite{frobenius1906reellen} showed that
the $x^{2}$-measure on a finite group is given by
\[
\frac{1}{\left|G\right|}\sum_{\psi\in\mathrm{Irr}\left(G\right)}\fs_{\psi}\cdot\psi,
\]
where $\fs_{\psi}$ is the Frobenius-Schur indicator of $\psi$:
\[
\fs_{\psi}=\begin{cases}
1 & \psi~\mathrm{is~afforded~by~a~real~representation}\\
-1 & \psi~\mathrm{is~real~but~is~not~afforded~by~any~real~representation}\\
0 & \psi~\mathrm{is~not~real}.
\end{cases}
\]
The statement is equivalent to that the expected value of an irreducible
character $\psi$ under the $x^{2}$-measure is $\fs_{\psi}$. In
fact, these two results hold for any compact group $G$ and can be
generalized to any surface word:
\begin{thm}[Frobenius, Frobenius-Schur]
\label{thm:Frobenius} Let $G$ be a compact group and $\psi$ an
irreducible character of $G$. Then,
\begin{enumerate}
\item For $w=\left[x_{1},y_{1}\right]\cdots\left[x_{g},y_{g}\right]$, the
expected value of $\psi$ under the $w$-measure is $\frac{1}{\psi\left(e\right)^{2g-1}}$.
\item For $w=x_{1}^{2}\cdots x_{g}^{2}$, the expected value of $\psi$
under the $w$-measure is $\frac{\left(\fs_{\psi}\right)^{g}}{\psi\left(e\right)^{g-1}}$.
\end{enumerate}
\end{thm}

\selectlanguage{american}%
Of course, here $x_{1},\ldots,x_{g},y_{1},\ldots,y_{g}$ are distinct
letters. See Section \ref{sec:Prelimiaries} for some details about
the proof.

\subsection*{Do word measures determine the word?}

For $w_{1},w_{2}\in\F_{r}$ write $w_{1}\auteqr w_{2}$ if there is
$\theta\in\mathrm{Aut}\big(\F_{r}\big)$ with $\theta\left(w_{1}\right)=w_{2}$.
It is easy to see that applying elementary Nielsen transformations
on a word does not change the measures it induces on groups (e.g.,
see \cite[Fact 2.5]{MP}), and thus
\begin{fact}
If $w_{1}\auteqr w_{2}$ then $w_{1}$ and $w_{2}$ induce the same
measure on every compact group.
\end{fact}

For example, $xyxy^{-1}\auteq x^{2}y^{2}$ and so the expected value
of $\psi$ under the $xyxy^{-1}$-measure is $\frac{\left(\fs_{\psi}\right)^{2}}{\psi\left(e\right)}$.
More generally, every non-trivial word $w\in\F_{r}$ in which every
letter appears exactly twice is mapped by $\mathrm{Aut}\big(\F_{r}\big)$
to one of the words in Theorem \ref{thm:Frobenius}. By the classification
of surfaces, the suitable word is determined by the homeomorphism
type of the surface obtained from gluing the sides of a $\left|w\right|$-gon
according to the letters of $w$ (see, for instance, \cite[Chapter 1.3]{stillwell2012classical}).
The expected value of irreducible characters is then described by
the suitable case in the theorem.

Several mathematicians, including A.~Amit, T.~Gelander, A.~Lubotzky,
A.~Shalev and U.~Vishne, conjecture that the converse is also true
and that every other pair of words is ``measurably separable'':
\begin{conjecture}
\label{conj:shalev minus}Let $w_{1},w_{2}\in\F_{r}$. If $w_{1}$
and $w_{2}$ induce the same measure on every compact group, then
$w_{1}\auteqr w_{2}$.
\end{conjecture}

This conjecture appears in the literature in a stronger form, where
$w_{1}$ and $w_{2}$ are only assumed to induce the same measure
on every \emph{finite }group -- see \cite[Question 2.2]{Amit2011},
\cite[Conjecture 4.2]{Shalev2013} and \cite[Section 8]{PP15}. 

Conjecture \ref{conj:shalev minus} seems to be extremely challenging.
Our focus here, instead, is on special cases, where $w_{1}$ is some
fixed word. A case which attracted considerable attention was that
of primitive words, namely the $\mathrm{Aut}\big(\F_{r}\big)$-orbit
containing the free generators of $\F_{r}$. This special case was
settled by the second author and Parzanchevski \cite{PP15}, who showed
that $w$ induces the uniform measure on the symmetric group $S_{N}$
for all $N$ if and only if $w$ is primitive. To the best of our
knowledge, the only $\Aut\big(\F_{r}\big)$-orbits for which the expected
value of irreducible characters have a simple explicit formula for
every compact group, are the primitive case (where all characters
but the trivial one have expectation zero) and surface words as in
Theorem \ref{thm:Frobenius} -- see also \cite{parzanchevski2014fourier}
and the references therein. In this sense, surface words are a natural
next case to consider. And, indeed, we settle Conjecture \ref{conj:shalev minus}
when $w_{1}$ is a surface word:
\begin{thm}
\label{thm:surface words are separable}Let $w\in\F_{r}$.
\begin{enumerate}
\item If $w$ induces the same measure as $\left[x_{1},y_{1}\right]\cdots\left[x_{g},y_{g}\right]$
on every compact group, then ($r\ge2g$, and) $w\auteqr\left[x_{1},y_{1}\right]\cdots\left[x_{g},y_{g}\right]$.
\item If $w$ induces the same measure as $x_{1}^{2}\cdots x_{g}^{2}$ on
every compact group, then ($r\ge g$, and) $w\auteqr x_{1}^{2}\cdots x_{g}^{2}$.
\end{enumerate}
\end{thm}

\begin{rem}
Let us mention another new result in the same spirit as Theorem \ref{thm:surface words are separable}.
Let $w_{0}$ be either a primitive power, say $w_{0}=x_{1}^{m}$,
or any power of the simple commutator $w_{0}=\left[x,y\right]^{m}$.
In a forthcoming paper \cite{Hanani}, Hanani, Meiri and the second
author show that if a word $w$ induces the same measure as $w_{0}$
on every \emph{finite} group, then $w\auteqr w_{0}$. In the case
of the simple commutator $\left[x,y\right]$, this strengthens Theorem
\ref{thm:surface words are separable}, as it only relies on measures
on finite groups. On the other hand, unlike the current paper where
we use specific families of groups ($\U\big(N\big)$ and generalized
permutation groups), the finite groups relied upon in \cite{Hanani}
are not explicit.
\end{rem}

\subsection*{Word measures on $\protect\U\left(N\right)$, on $\protect\O\left(N\right)$,
and on generalized symmetric groups}

Our proof of Theorem \ref{thm:surface words are separable} relies
on the measures induced by words on unitary groups, on orthogonal
groups, and on generalized symmetric groups. While we study the former
two in separate papers, we study the latter one here. In fact, it
is enough to consider the expected value of the standard character,
namely, the expected value of the trace, in all three families of
groups, of the ``defining representation''\footnote{Similarly, the standard character of the symmetric group was sufficient
for the result in \cite{PP15} considering the primitive orbit.}. We denote the expected value of the trace of a $w$-random element
in a matrix group $G$ by $\trw\left(G\right)$\marginpar{$\protect\trw\left(G\right)$}.

To present our results, we introduce some notation. Recall that given
a free group $\F$, the commutator subgroup $\left[\F,\F\right]$
is the kernel of the homomorphism $\F\twoheadrightarrow\mathbb{Z}^{\mathrm{rank}\left(\F\right)}$
mapping every generator to a different element of the standard generating
set of $\mathbb{Z}^{\mathrm{rank}\left(\F\right)}$. Similarly, \foreignlanguage{english}{for
$m\in\mathbb{Z}_{\ge2}$, let $C_{m}\stackrel{\mathrm{def}}{=}\nicefrac{\mathbb{Z}}{m\mathbb{Z}}$
be the cyclic group of order $m$}, and denote by\marginpar{$K_{m}\left(\protect\F\right)$}
\[
K_{m}\left(\F\right)\stackrel{\mathrm{def}}{=}\ker\left(\F\twoheadrightarrow C_{m}^{~\mathrm{rank}\left(\F\right)}\right)
\]
the kernel of the homomorphism $\F\twoheadrightarrow C_{m}^{~\mathrm{rank}\left(\F\right)}$
mapping every generator to a different element of a standard generating
set of $C_{m}^{~\mathrm{rank}\left(\F\right)}$. Note that even though
the homomorphism $\F\twoheadrightarrow C_{m}^{~\mathrm{rank}\left(\F\right)}$
depends on the choice of basis of $\F$, its kernel does not, and
it consists of all words where $m$ divides the total exponent of
every generator. For efficiency of presenting our results, we also
denote $K_{\infty}\left(\F\right)\stackrel{\mathrm{def}}{=}\left[\F,\F\right]$\marginpar{$K_{\infty}\left(\protect\F\right)$}.

Is it a standard fact that $w\in\F$ is a product of squares if and
only if $w\in K_{2}\left(\F\right)$. Likewise, $w$ is a product
of commutators if and only if $w\in K_{\infty}\left(\F\right)=\left[\F,\F\right]$.\footnote{These facts about $K_{2}\left(\F\right)$ and $K_{\infty}\left(\F\right)$
do not generalize to $K_{m}\left(\F\right)$ for $2\le m<\infty$.} We can now extract from \cite{MP-Un,MP-On} the results we need here.
\begin{defn}
\label{def:cl}Let $w\in\F_{r}$. The \emph{commutator length} of
$w\in\F_{r}$ is defined as\marginpar{$\protect\cl\left(w\right)$}
\[
\cl\left(w\right)\stackrel{\mathrm{def}}{=}\min\left\{ g\,\middle|\,\begin{gathered}\exists u_{1},v_{1},\ldots,u_{g},v_{g}\in\F_{r}~\mathrm{s.t.}\\
w=\left[u_{1},v_{1}\right]\cdots\left[u_{g},v_{g}\right]
\end{gathered}
\right\} .
\]
In particular, if $w\notin\left[\F_{r},\F_{r}\right]$, then $\cl\left(w\right)=\infty$.
Similarly, the \emph{square length} of $w\in\F_{r}$ is defined as\marginpar{$\protect\sql\left(w\right)$}
\[
\sql\left(w\right)\stackrel{\mathrm{def}}{=}\min\left\{ g\,\middle|\,\begin{gathered}\exists u_{1},\ldots,u_{g}\in\F_{r}~\mathrm{s.t.}\\
w=u_{1}^{2}\cdots u_{g}^{2}
\end{gathered}
\right\} .
\]
In particular, if $w\notin K_{2}\left(\F_{r}\right)$, then $\sql\left(w\right)=\infty$.
\end{defn}

\begin{thm}
\cite[Corollary 1.8]{MP-Un}\label{thm:Un} Fix $w\in\F_{r}$ and
consider the measure it induces on the unitary groups $\U\left(N\right)$.
The expected trace of a $w$-random unitary matrix in $\U\left(N\right)$
satisfies 
\[
\trw\left(\U\left(N\right)\right)=O\left(N^{1-2\cdot\cl\left(w\right)}\right).
\]
\end{thm}

\begin{thm}
\cite[Corollary 1.10]{MP-On}\label{thm:On} Fix $w\in\F_{r}$ and
consider the measure it induces on the orthogonal groups $\O\left(N\right)$.
The expected trace of a $w$-random orthogonal matrix in $\O\left(N\right)$
satisfies 
\[
\trw\left(\O\left(N\right)\right)=O\left(N^{1-\min\left(\sql\left(w\right),2\cl\left(w\right)\right)}\right).
\]
\end{thm}

In the current paper we obtain similar results for generalized symmetric
groups. Specifically, \foreignlanguage{english}{let $\ssn$}\marginpar{\selectlanguage{english}%
$\protect\ssn$\selectlanguage{american}%
}\foreignlanguage{english}{ denote the wreath product of ${\cal S}^{1}$
with $S_{N}=\mathrm{Sym}\left(N\right)$, namely, this is the subgroup
of $\mathrm{GL}_{N}\left(\mathbb{C}\right)$ consisting of matrices
with exactly one non-zero entry in every row and column and all non-zero
entries having absolute value $1$. Likewise, for $m\in\mathbb{Z}_{\ge2}$,
let $\cmsn$}\marginpar{$\protect\cmsn$}\foreignlanguage{english}{
be the wreath product}\footnote{\selectlanguage{english}%
The group $\cmsn$ is sometimes denoted $S\left(m,N\right)$ -- see,
for example,\\
https://en.wikipedia.org/wiki/Generalized\_symmetric\_group.\selectlanguage{american}%
}\foreignlanguage{english}{ of $C_{m}$, the cyclic group of order
$m$, with $S_{N}$. This is the subgroup of $\ssn$ where all non-zero
entries are $m$-th roots of unity. Note that when $m=2$, the group
$C_{2}\wr S_{N}$ is the signed symmetric group, known also as the
hyper-octahedral group or the Coxeter group of type $B_{N}=C_{N}$.}

The first observation we make is that the expected value under the
$w$-measure of the trace of any of these groups is given by a rational
function in $N$:
\begin{lem}
\label{lem:rational function}Fix $w\in\F_{r}$ and let $G\left(N\right)=\cmsn$
for some fixed $m\in\mathbb{Z}_{\ge2}$ or $G\left(N\right)=\ssn$.
Then there is some rational function $f\in\mathbb{Q}\left(x\right)$,
such that for every large enough $N$, $\trw\left(G\left(N\right)\right)=f\left(N\right)$.
\end{lem}

For example, if $w=x^{2}y^{3}x^{2}y^{-1}$ and $m=2$, then $\trw\left(\text{\ensuremath{\ctsn}}\right)=\frac{3N-4}{N\left(N-1\right)}$
for all $N\ge2$. The proof of this lemma appears in Section \ref{subsec:Rational-expressions-and-proofs}.
We stress that a statement of this sort is not surprising: a similar
statement is known to hold for various families of characters of the
groups $G\left(N\right)$ when $G\left(N\right)=S_{N}$ \cite{nica1994number,Linial2010},
when $G\left(N\right)=\U\left(N\right)$ \cite{MP-Un}, when $G\left(N\right)=\O\left(N\right)$
or ${\cal S}p\left(N\right)$ \cite{MP-On}, or when $G\left(N\right)=\mathrm{GL}_{N}\left(\mathbb{F}_{q}\right)$
is the general linear group over the finite field $\mathbb{F}_{q}$
\cite{West}.

Our main result with respect to word measures on these generalized
symmetric groups revolves around the leading term of the rational
expression from Lemma \ref{lem:rational function}. The exponent of
the leading term is described by the number in the following definition.
\begin{defn}
\label{def:mth primitivity rank}Let $w\in\F_{r}$ and $m\in\mathbb{Z}_{\ge2}$
or $m=\infty$. Denote\marginpar{$\protect\ecm\left(w\right)$}\emph{
}
\begin{equation}
\ecm\left(w\right)\stackrel{\mathrm{def}}{=}1-\min\left\{ \mathrm{rank}\left(H\right)\,\middle|\,\begin{gathered}H\le\F_{r}\\
w\in K_{m}\left(H\right)
\end{gathered}
\right\} .\label{eq:def of ECM}
\end{equation}
If the set in the right hand side of \eqref{eq:def of ECM} is empty,
we set $\ecm\left(w\right)=-\infty$.
\end{defn}

In words, we look for subgroups $H\le\F_{r}$ of smallest rank such
that $K_{m}\left(H\right)$ contains $w$, and take their Euler characteristic
which equals $1-\mathrm{rank}\left(H\right)$. It is easy to see that
$K_{m}\left(H\right)\le K_{m}\left(\F\right)$ whenever $H\le\F$,
and so if $\ecm\left(w\right)$ is not $-\infty$, it is in fact at
least $1-r$. Thus, in $\F_{r}$, the function $\ecm$ takes values
in $\left\{ 1,0,-1,\ldots,1-r\right\} \cup\left\{ -\infty\right\} $,
where $\ecm\left(w\right)=1\Leftrightarrow w=1$ and $\ecm\left(w\right)=0$
if and only if $w=u^{m}$ for some $1\ne u\in\F_{r}$.

To illustrate, $w=x^{2}y^{3}x^{2}y^{-1}$ is not a proper power, so
$\ecm\left(w\right)<0$. For $m=2$, $w\in K_{2}\left(H\right)$ for
$H=\left\langle x,y\right\rangle $ (as well as for $H=\left\langle x^{2},y\right\rangle $
and for $H=\left\langle x^{2}y,y^{2}\right\rangle $) and so $\ect\left(w\right)=-1$
. For $m\ge3$ or $m=\infty$, $\chi_{m}\left(w\right)=-\infty$.
As another example, consider the orientable surface word $w=\left[x_{1},y_{1}\right]\cdots\left[x_{g},y_{g}\right]$.
Then $w\in K_{m}\left(\left\langle x_{1},y_{1},\ldots,x_{g},y_{g}\right\rangle \right)$
for every $m\in\mathbb{Z}_{\ge2}\cup\left\{ \infty\right\} $, and
one can show that $\ecm\left(w\right)=1-2g$.
\begin{thm}
\label{thm:trace-of-generalized-permutations - baby version}Let $w\in\F_{r}$
and $m\in\mathbb{Z}_{\ge2}$ or $m=\infty$. If $m\in\mathbb{Z}_{\ge2}$,
consider a $w$-random matrix in the group $G\left(N\right)=\cmsn$,
and if $m=\infty$ consider a $w$-random matrix in $G\left(N\right)=S^{1}\wr S_{N}$.
Then 
\begin{equation}
\trw\left(G\left(N\right)\right)=C\cdot N^{\ecm\left(w\right)}+O\left(N^{\ecm\left(w\right)-1}\right),\label{eq:trace-og-generalized-permutations-baby-version}
\end{equation}
where $C$ is a natural number counting the number of subgroups $H\le\F_{r}$
with $\ecm\left(w\right)=1-\mathrm{rank}\left(H\right)$ and $w\in K_{m}\left(H\right)$.
In particular, $\trw\left(G\left(N\right)\right)$ vanishes if $\ecm\left(w\right)=-\infty$.
\end{thm}

Namely, the coefficient $C$ in \eqref{eq:trace-og-generalized-permutations-baby-version}
counts the number of the subgroups $H$ demonstrating the value of
$\text{\ensuremath{\ecm\left(w\right)}}$ \foreignlanguage{english}{determined
in \eqref{eq:def of ECM}. }This number is always finite -- see Section
\ref{subsec:Rational-expressions-and-proofs}. In fact, we have a
more detailed result which is required for the proof of Theorem \ref{thm:surface words are separable}
-- see Theorem \ref{thm:two min rank} below. Theorem \ref{thm:trace-of-generalized-permutations - baby version}
is similar in spirit to \cite[Theorem 1.8]{PP15}, where $\trw\left(S_{N}\right)$
is analyzed. The group $S_{N}$ can be regarded as the $m=1$ case
in the current terminology. The analog there of $\ecm\left(w\right)$
is the ``primitivity rank'' of $w$. Moreover, the more detailed
version of Theorem \ref{thm:trace-of-generalized-permutations - baby version},
Theorem \ref{thm:two min rank} below, relies on much of the analysis
from \cite{PP15}. A crucial difference between the current groups
and $S_{N}$ is that the standard defining $N$-dimensional representation
is reducible for $S_{N}$ but irreducible for the groups considered
in Theorem \ref{thm:trace-of-generalized-permutations - baby version}.
We further explain these connections in Sections \ref{sec:Prelimiaries}
and \ref{subsec:Rational-expressions-and-proofs}.

\subsection*{Overview of the proof}

The proof of Theorem \ref{thm:surface words are separable} uses both
the measures on the classical groups $\U\left(N\right)$ and $\O\left(N\right)$,
and the measures on generalized symmetric groups. The roles they play
are somewhat complement. Let us illustrate these complementing roles
by considering the commutator length of a word. Let $w\in\F_{r}$,
and consider the measure it induces on $\U\left(N\right)$. If $\trw\left(\U\left(N\right)\right)=\Theta\left(N^{1-2g}\right)$,
Theorem \ref{thm:Un} yields an upper bound on the commutator length:
$\cl\left(w\right)\le g$.

In contrast, if $w=\left[u_{1},v_{1}\right]\cdots\left[u_{\cl\left(w\right)},v_{\cl\left(w\right)}\right]$
then $w\in K_{\infty}\left(H\right)=\left[H,H\right]$, where\linebreak{}
$H=\left\langle u_{1},v_{1},\ldots,u_{\cl\left(w\right)},v_{\cl\left(w\right)}\right\rangle $
which has rank at most $2\cl\left(w\right)$ and thus $\eci\left(w\right)\ge1-\mathrm{rank}\left(H\right)\ge1-2\cl\left(w\right)$.
Hence if $\trw\left(\ssn\right)=\Theta\left(N^{1-2g}\right)$, we
deduce the lower bound $\cl\left(w\right)\ge g$.

If $w$ induces the same measure as $\left[x_{1},y_{1}\right]\cdots\left[x_{g},y_{g}\right]$
on every compact group, then, in particular,\linebreak{}
$\trw\left(\U\left(N\right)\right)=\trw\left(\ssn\right)=N^{1-2g}$.
The preceding two paragraphs then show that $\cl\left(w\right)=g$.
Moreover, they show the group $H$ from the preceding paragraph has
rank exactly $2g$, and so $u_{1},v_{1},\ldots,u_{\cl\left(w\right)},v_{\cl\left(w\right)}$
are free words, namely, there is no non-trivial relation on them.
Together with Theorem \ref{thm:two min rank} below (a strengthening
of Theorem \ref{thm:trace-of-generalized-permutations - baby version}),
it is possible to deduce that $u_{1},v_{1},\ldots,u_{\cl\left(w\right)},v_{\cl\left(w\right)}$
are, in fact, part of a basis of $\F_{r}$, and therefore $w\auteq\left[x_{1},y_{1}\right]\cdots\left[x_{g},y_{g}\right]$.

The paper is organized as follows. Section \ref{sec:Prelimiaries}
contains some background regarding measures induced by surface words,
as well as background regarding word measures on $S_{N}$ and some
results from \cite{PP15} we use here. It also introduces the notions
of algebraic extensions and of core graphs. In Section \ref{sec:Expected-trace-in-generalized-symmetric-groups}
we analyze word measures on generalized symmetric groups and prove
Lemma \ref{lem:rational function}, Theorem \ref{thm:trace-of-generalized-permutations - baby version},
and the stronger Theorem \ref{thm:two min rank}. We prove Theorem
\ref{thm:surface words are separable} in Section \ref{sec:Surface-Words-proof}
and conclude with some open questions in Section \ref{sec:Open-Questions}.

\subsection*{Notation}

We use the following asymptotic notation. Let $f,g\colon\mathbb{Z}_{\ge1}\to\mathbb{R}$
be two functions defined on the positive integers. We write 
\begin{itemize}
\item $f=O(g)$ if there is a constant $C>0$ such that $\left|f(n)\right|\le C\cdot g(n)$
for every large enough $n$,
\item $f=\Omega(g)$ if there is a constant $C>0$ such that $\left|f(n)\right|\ge C\cdot g(n)$
for every large enough $n$, and
\item $f=\Theta(g)$ if both $f=O(g)$ and $f=\Omega(n)$.
\end{itemize}
\selectlanguage{english}%

\section{Preliminaries\label{sec:Prelimiaries}}
\selectlanguage{american}%

\subsection*{Measures induced by surface words}

\selectlanguage{english}%
We begin this section with some remarks regarding the proof of Theorem
\ref{thm:Frobenius}. \foreignlanguage{american}{We have already mentioned
a reference \cite{frobenius1896gruppencharaktere} for the case where
$G$ is finite and $w=\left[x,y\right]$. In fact, this case is at
the level of an exercise for an arbitrary compact group $G$, as long
as one is aware of the following classical facts: matrix coefficients
of unitary realizations of all irreducible representations of a compact
group form an orthogonal basis for the space of complex functions
on $G$, and the $L^{2}$-norm of a matrix coefficient of a $d$-dimensional
irreducible representation is $\frac{1}{d}$.}

\selectlanguage{american}%
The case of $w=x^{2}$ and $G$ finite was first proved in \cite{frobenius1906reellen}.
For an English proof see \cite[Chapter 4]{isaacs1994character}. Although
the book \cite{isaacs1994character} deals with finite groups, this
proof applies just as well to general compact groups. 

Finally, for $g\ge2$, note that when the letters appearing in $w_{1}$
are distinct from those in $w_{2}$, then the $w_{1}w_{2}$-measure
on $G$ is the convolution of the $w_{1}$-measure and the $w_{2}$-measure,
and using the fact that a $w$-measure is always invariant under conjugation,
we get that $\mathbb{E}_{w_{1}w_{2}}\left(\psi\right)=\frac{1}{\psi\left(e\right)}\mathbb{E}_{w_{1}}\left(\psi\right)\cdot\mathbb{E}_{w_{2}}\left(\psi\right)$
for every irreducible character $\psi$ of $G$. This explains the
complete statement of Theorem \ref{thm:Frobenius}. See also \cite{parzanchevski2014fourier}
and the references therein.

\subsection*{Expected traces in $S_{N}$}

\selectlanguage{english}%
Next, we extract some terminology and results from \cite{PP15} that
are needed here. That paper analyzes $\trw\left(S_{N}\right)$, the
expected trace of a $w$-random permutation in $S_{N}$, where the
permutation is thought of as an $N\times N$ matrix. In other words,
it studies the expected number of fixed points in a $w$-random permutation.
We remark that word measures on $S_{N}$ alone do not suffice to establish
Theorem \ref{thm:surface words are separable}: all irreducible characters
of $S_{N}$ are afforded by real representations, and so the words
$\left[x,y\right]$ and $x^{2}y^{2}$ induce the exact same measure
on $S_{N}$ for all $N$.

Let $\left|w\right|$ denote the length of the reduced form of $w$.
A first observation in the study of $\trw\left(S_{N}\right)$, going
back to Nica \cite{nica1994number}, is that for $N\ge\left|w\right|$,
$\trw\left(S_{N}\right)$ is a rational expression in $N$. Unlike
the other groups mentioned above, this $N$-dimensional representation
of $S_{N}$ is reducible: it is the sum of the trivial representation
and an $\left(N-1\right)$-dimensional irreducible representation.
It is thus not surprising that the rational expression for $\trw\left(S_{N}\right)$
has a contribution of $1$ coming from the trivial representation,
and the interesting part is the deviation from $1$. This deviation
is measured by the ``primitivity rank'' of a word $w\in\F_{r}$,
denoted $\pi\left(w\right)$, which was first introduced in \cite{Puder2014}.
Recall that an element of a free group is called \emph{primitive}
if it belongs to some basis (free generating set). The primitivity
rank of $w\in\F_{r}$ is the following number:
\begin{equation}
\pi\left(w\right)\stackrel{\mathrm{def}}{=}\min\left\{ \mathrm{rank}\left(H\right)\,\middle|\,w\in H\le\F_{r},~w~\mathrm{is~not~primitive~in}~H\right\} .\label{eq:primitivity rank words}
\end{equation}
The functions $\ecm\left(w\right)$ defined above are closely related
to $\pi\left(w\right)$. In fact, one can give a single definition
which applies to all these functions simultaneously. Indeed, for $m\in\mathbb{Z}_{\ge1}\cup\left\{ \infty\right\} $
define 
\[
\ecm'\left(w\right)=1-\min\left\{ \mathrm{rank}\left(H\right)\,\middle|\,H\le\F_{r},w\in K_{m}\left(H\right),~w~\mathrm{is~not~primitive~in}~H\right\} .
\]
Now $\chi_{1}'\left(w\right)=1-\pi\left(w\right)$ as $K_{1}\left(H\right)=H$,
and for $m\ne1$, $\ecm'\left(w\right)=\ecm\left(w\right)$ because
all elements of $K_{m}\left(H\right)$ are automatically non-primitive
in $H$. These different functions of words also share some properties.
For instance, for all $m\in\mathbb{Z}_{\ge1}\cup\left\{ \infty\right\} $,
$\ecm'\left(w\right)$ takes values in $\left\{ 1,0,-1,\ldots,1-r\right\} \cup\left\{ -\infty\right\} $
-- this\footnote{To be precise, $\chi_{\infty}$ is never zero: a cyclic group has
a trivial commutator subgroup.} was explained above for $m\ne1$, and for $m=1$, this is \cite[Corollary 4.2]{Puder2014}.
The role of $\pi\left(w\right)$ in the study of $\trw\left(S_{N}\right)$
is also analogous to the role of $\ecm\left(w\right)$ in Theorem
\ref{thm:trace-of-generalized-permutations - baby version}:
\begin{thm}
\label{thm:pp15, trace of words}\cite[Theorem 1.8]{PP15} Let $w\in\F_{r}$.
Then 
\[
\trw\left(S_{N}\right)=1+C\cdot N^{1-\pi\left(w\right)}+O\left(N^{-\pi\left(w\right)}\right),
\]
where $C\in\mathbb{Z}_{\ge1}$ is the number of subgroups $H\in\F_{r}$
of rank $\pi\left(w\right)$ containing $w$ as a non-primitive element.\\
In particular, $\trw\left(S_{N}\right)\equiv1$ for all $N$ if and
only if $\pi\left(w\right)=\infty$, which holds if and only if $w$
is primitive.
\end{thm}

\subsection*{Random subgroups in $S_{N}$}

The results in \cite{PP15} apply not only to random elements of $S_{N}$
with measures induced by words $w\in\F_{r}$, but more generally,
to random subgroups of $S_{N}$ with measures induced by subgroups
$H\le\F_{r}$. Given $H$, sample a random subgroup of $S_{N}$ by
choosing a homomorphism $\varphi\in\mathrm{Hom}\left(\F_{r},S_{N}\right)$
uniformly at random and considering $\varphi\left(H\right)\le S_{N}$.
When $H=\left\langle w\right\rangle $, the resulting random subgroup
is the one generated by a $w$-random permutation.

\selectlanguage{american}%
If $H\le J$ are free groups, we say that $J$ is a \emph{free extension}
of $H$, or that $H$ is a \emph{free factor }of $J$, and denote
$H\ff J$\marginpar{$H\protect\ff J$}, if some (and therefore every)
basis of $H$ can be extended to a basis of $J$. Clearly, for $w\ne1$,
$w$ is primitive in $J$ if and only if $\left\langle w\right\rangle \ff J$.
Hence, the following notion of primitivity rank for subgroups generalizes
\eqref{eq:primitivity rank words}. For $H\le\F_{r}$, the primitivity
rank of $H$ is defined to be
\[
\pi\left(H\right)\stackrel{\mathrm{def}}{=}\min\left\{ \mathrm{rank}\left(J\right)\,\middle|\,H\le J\le\F_{r},H~\mathrm{\text{is~not~a~free~factor~of ~J}}\right\} .
\]
We can now state the more general form of Theorem \ref{thm:pp15, trace of words}.
\begin{thm}
\label{thm:PP15 subgroups}\cite[Theorem 1.8]{PP15} Let $H\le\F_{r}$
be a finitely generated subgroup, and let $\varphi\in\mathrm{Hom}\left(\F_{r},S_{N}\right)$
be a uniformly random homomorphism. The expected number of points
in $\left\{ 1,\ldots,N\right\} $ fixed by all elements of the subgroup
$\varphi\left(H\right)$ is
\[
N^{1-\mathrm{rank}\left(H\right)}+C\cdot N^{1-\pi\left(H\right)}+O\left(N^{-\pi\left(H\right)}\right),
\]
where $C$ is the number of subgroups $J\le\F_{r}$ satisfying $\mathrm{rank}\left(J\right)=\pi\left(H\right)$
and containing $H$ but not as a free factor. \\
In particular, this value is $N^{1-\mathrm{rank}\left(H\right)}$
for all $N$ if and only if $\pi\left(H\right)=\infty$, which holds
if and only if $H\ff\F_{r}$.
\end{thm}

\subsection*{Algebraic extensions}

\selectlanguage{english}%
We now describe the notion of \emph{algebraic extensions }in free
groups which is used in Section \ref{subsec:The-second-term} below.
Let $\F$ be a free group and $H,J\le\F$ two subgroups. We call $J$
an algebraic extension of $H$, denoted $H\alg J$\marginpar{$H\protect\alg J$},
if and only if $H\le J$ and there is no intermediate proper free
factor of $J$, namely, if whenever $H\le M\ff J$, we must have $M=J$.
To give a sense of this notion, we mention some of its properties:
algebraic extensions form a partial order on the set of subgroups
of $\F$; for every extension of free groups $H\le J$ there is a
unique intermediate subgroup $A$ satisfying $H\alg A\ff J$; and
every finitely generated subgroup $H$ of $\F$ has finitely many
algebraic extensions. See the survey \cite{miasnikov2007algebraic}
or Section 4 of \cite{PP15} for more details. 

In the language of algebraic extensions, $\pi\left(H\right)$ is the
smallest rank of a proper algebraic extension of $H$, and $\pi\left(w\right)$
is the smallest rank of a proper algebraic extension of $\left\langle w\right\rangle $.

\subsection*{Core graphs}

Recall that we have a fixed basis $x_{1},\ldots,x_{r}$ for $\F_{r}$.
Call it $X$. Associated with every (finitely generated) subgroup
$H$ of $\F_{r}$ is a rooted, directed and edge-labeled (finite)
graph, where the edges are labeled by $x_{1},\ldots,x_{r}$. This
graph, denoted $\Gamma_{X}\left(H\right)$\marginpar{$\Gamma_{X}\left(H\right)$},
is called the \emph{(Stallings) core graph }of $H$ and was introduced
in \cite{stallings1983topology}. It can be obtained from the Schreier
graph depicting the right action of $\F_{r}$ on $H\backslash\F_{r}$,
the right cosets of $H$ in $\F_{r}$, by trimming all ``hanging
trees''. For more details we refer the reader to \cite[Section 3]{PP15}.
We illustrate the concept in Figure \ref{fig:first_core_graph}.

\FigBesBeg 
\begin{figure}[h]
\centering{}%
\begin{minipage}[t]{0.5\columnwidth}%
\[
\xymatrix{\otimes\ar[rr]^{x_{1}} &  & \bullet\\
\\
\bullet\ar[rr]^{x_{1}}\ar[uu]_{x_{2}} &  & \bullet\ar[uull]_{x_{1}}\ar[uu]_{x_{2}}
}
\]
\end{minipage}\caption{\label{fig:first_core_graph} The core graph $\Gamma_{X}\left(H\right)$
where $H=\left\langle x_{1}x_{2}^{-1}x_{1},x_{1}^{-2}x_{2}\right\rangle \protect\leq\protect\F_{2}$.}
\end{figure}
\FigBesEnd 

Let us mention here a few basic facts about core graphs and some further
notations that we will need below. The elements of $H$ correspond
exactly to the non-backtracking closed paths at the root of $\Gamma_{X}\left(H\right)$.
The labels and directions of the edges give rise to a graph-morphism
to the bouquet of $r$ directed loops, labeled by $x_{1},\ldots,x_{r}$,
and this morphism is always an immersion. In other words, every vertex
of $\Gamma_{X}\left(H\right)$ has at most one outgoing edge with
a given label, and at most one incoming edge with a given label.

A morphism of rooted, directed and edge labeled graphs from $\Gamma_{X}\left(H\right)$
to $\Gamma_{X}\left(J\right)$ exists if and only if $H\le J$. When
this morphism is surjective, we say that $H$ ``\emph{$X$-covers}''
$J$, and denote $H\covers J$\marginpar{\selectlanguage{american}%
$\protect\covers$\selectlanguage{english}%
}. This relation constitutes a partial order on the set of finitely
generated subgroups of $\F_{r}$, a partial order which depends on
the choice of basis $X$. The easiest way to explain why there is
a rational expression for $\trw\left(S_{N}\right)$ and, moreover,
to compute this formula explicitly, is by considering the finite set\marginpar{$\protect\XF H$}
\[
\XF H\stackrel{\mathrm{def}}{=}\left\{ H\le\F_{r}\,\middle|\,\left\langle w\right\rangle \covers H\right\} 
\]
of subgroups which are $X$-covered by the subgroup $\left\langle w\right\rangle $
(see \cite[Section 5]{Puder2014}). We shall use these graphs below
to prove Lemma \ref{lem:rational function} and Theorem \ref{thm:trace-of-generalized-permutations - baby version}.

\section{\label{sec:Expected-trace-in-generalized-symmetric-groups}Expected
trace in generalized symmetric groups}

\subsection{\label{subsec:Rational-expressions-and-proofs}Rational expressions
and their leading term}

\selectlanguage{american}%
Fix $w\in\F_{r}$ and let $G\left(N\right)$ be one of the groups
$\cmsn$ ($m\in\mathbb{Z}_{\ge2}$), $\ssn$, or merely $S_{N}$,
realized as $N\times N$ complex matrices. If $w=x_{i_{1}}^{\varepsilon_{1}}x_{i_{2}}^{\varepsilon_{2}}\cdots x_{i_{\ell}}^{\varepsilon_{\ell}}\in\F_{r}$
(here $i_{1},\ldots,i_{\ell}\in\left\{ 1,\ldots,r\right\} $ and $\varepsilon_{1},\ldots,\varepsilon_{\ell}\in\left\{ \pm1\right\} $),
we analyze the following expression:
\begin{equation}
\trw\left(G\left(N\right)\right)=\int_{A_{1},\ldots,A_{r}\in G\left(N\right)}\mathrm{tr}\left(w\left(A_{1},\ldots,A_{r}\right)\right)=\sum_{j_{1},\ldots,j_{\ell}=1}^{N}\int_{A_{1},\ldots,A_{r}\in G\left(N\right)}\left[A_{i_{1}}^{\varepsilon_{1}}\right]_{j_{1},j_{2}}\left[A_{i_{2}}^{\varepsilon_{2}}\right]_{j_{2},j_{3}}\cdots\left[A_{i_{\ell}}^{\varepsilon_{\ell}}\right]_{j_{\ell},j_{1}},\label{eq:trw(GN) first step}
\end{equation}
where $A_{1},\ldots,A_{r}$ are independent Haar-uniform elements
of $G\left(N\right)$. In all cases considered but $\ssn$ this is
the uniform measure on $G\left(N\right)$. The Haar measure on $\ssn$
is given by a uniform distribution on $S_{N}$ to determine the non-zero
entries and independent Lebesgue measure on the unit circle for every
non-zero entry of the matrix.

Consider an assignment of values in $\left\{ 1,\ldots,N\right\} $
to the indices $j_{1},\ldots,j_{\ell}$. Every assignment induces
a partition on $\left\{ 1,\ldots,\ell\right\} $, where two indices
$s$ and $t$ belong to the same block if and only if $j_{s}=j_{t}$.
Such a partition can be described by a rooted, directed, edge-labeled
graph as follows: the vertices correspond to the blocks in the partition
on $\left\{ 1,\ldots,\ell\right\} $, the root is the block containing
$1$, and for every $t=1,\ldots,\ell$ there is a directed edge labeled
$x_{i_{t}}$ connecting the block of $j_{t}$ with the block of $j_{\left(t+1\right)\mod\ell}$,
and directed towards the $j_{\left(t+1\right)\mod\ell}$ block if
$\varepsilon_{t}=1$ or towards the $j_{t}$ block in case $\varepsilon_{t}=-1$.
There is at most one $x_{i}$-edge directed from a vertex $u$ to
a vertex $v$. For example, if $w=x_{1}^{2}x_{2}^{2}$ and the assignment
is $\left(j_{1},j_{2},j_{3},j_{4}\right)=\left(1,1,3,5\right)$, the
graph is the following:

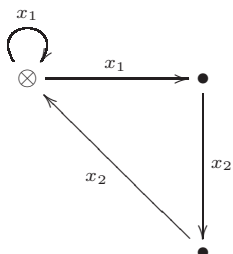
\selectlanguage{english}%
\begin{minipage}[t]{0.5\columnwidth}%
\[
\xymatrix{\otimes\ar[rr]^{x_{1}}\ar@(ul,ur)^{x_{1}} &  & \bullet\ar[dd]^{x_{2}}\\
\\
 &  & \bullet\ar[uull]^{x_{2}}
}
\]
\end{minipage}

However, this assignment contributes zero to the summation in \eqref{eq:trw(GN) first step}:
in all groups considered here, there is exactly one non-zero entry
in every column and every line, yet the assignment $\left(1,1,3,5\right)$
leads to the integral over the monomial $\left[A_{1}\right]_{1,1}\left[A_{1}\right]_{1,3}\left[A_{2}\right]_{3,5}\left[A_{2}\right]_{5,1}$,
which involves two entries from the same row of $A_{1}$ and is thus
identically zero. This happens exactly when the graph associated with
the assignment has a vertex with two out-going edges with the same
label, or a vertex with two incoming edges with the same label. This
shows that we can restrict our attention to assignments associated
with graphs which are core graphs. Moreover, these graphs are precisely
the graphs which are $X$-covered by $\left\langle w\right\rangle $,
namely the graphs $\Gamma_{X}\left(H\right)$ for $H\in\XF{\left\langle w\right\rangle }$.

We can now group together all assignments leading to the same core
graph $\Gamma_{X}\left(H\right)$, and see that the contribution of
all these assignments is given by a rational function in $N$ (which
depends on the family of groups we consider). Because the set $\XF H$
is finite, this leads to a rational expression for $\trw\left(G\left(N\right)\right)$
for families of generalized symmetric groups. The number of assignments
associated with a given $\Gamma_{X}\left(H\right)$ is $N\left(N-1\right)\cdots\left(N-\#V\left(H\right)+1\right)$,
where $\#V\left(H\right)$ denotes the number of vertices in $\Gamma_{X}\left(H\right)$.
The probability that the uniformly random $A_{i}\in G\left(N\right)$
has non-zero entries which correspond to a given assignment associated
with $\Gamma$ is precisely $\frac{1}{N\left(N-1\right)\cdots\left(N-\#E_{i}\left(H\right)+1\right)}$,
where $\#E_{i}\left(H\right)$ is the number of $x_{i}$-edges in
$\Gamma_{X}\left(H\right)$. Overall, for $H\in\XF{\left\langle w\right\rangle }$,
if $N\ge\#E_{i}\left(H\right)$ for all $i$, the contribution of
$H$ to \eqref{eq:trw(GN) first step} is\footnote{We use the notation $L_{H,\F_{r}}^{X}\left(N\right)$ which is used
for this expression in \cite{PP15}.}\marginpar{\selectlanguage{american}%
$L_{H,\protect\F_{r}}^{X}\left(N\right)$\selectlanguage{english}%
} 
\[
L_{H,\F_{r}}^{X}\left(N\right)\stackrel{\mathrm{def}}{=}\frac{N\left(N-1\right)\cdots\left(N-\#V\left(H\right)+1\right)}{\prod_{i=1}^{r}N\left(N-1\right)\cdots\left(N-\#E_{i}\left(H\right)+1\right)}
\]
times the expected value of the product of non-zero entries of $A_{1},\ldots,A_{r}$
involved in the monomial in \eqref{eq:trw(GN) first step}. In the
case of $S_{N}$, these non-zero entries are identically $1$, and
so, as depicted in \cite[Section 5]{Puder2014}, for all $N$ large
enough,
\begin{equation}
\trw\left(S_{N}\right)=\sum_{H\in\XF{\left\langle w\right\rangle }}L_{H,\F_{r}}^{X}\left(N\right).\label{eq:formula for fixed points of a word in S_N}
\end{equation}
For example, in the case of $w=x_{1}^{2}x_{2}^{2}$, there are precisely
$7$ subgroups in $\XF{\left\langle w\right\rangle }$, and the total
contribution is $1+\frac{1}{N-1}$, holding for $N\ge2$. The detailed
computation for $w=\left[x_{1},x_{2}\right]$ is depicted in \cite[Page 53]{Puder2014}.

\selectlanguage{american}%
For the other groups considered here, the non-zero entries are not
identically one and actually have zero expectation. So often, the
contribution of an assignment to \eqref{eq:trw(GN) first step} vanishes
even when the assignment does correspond to some core graph. For example,
in the case of the group $C_{3}\wr S_{N}$, an assignment gives a
non-zero contribution if and only if it corresponds to a core graph,
\emph{and }every entry is repeated in the monomial some multiple of
$3$ times, when we count with signs. E.g., if the entry $A_{3,4}$
appears in the monomial in \eqref{eq:trw(GN) first step}, it must
appear a total number of 0 times as in $\cdots A_{3,4}\cdots A_{3,4}^{~-1}\cdots$,
a total number of $-3$ times as in $\cdots A_{3,4}^{-1}\cdots A_{3,4}^{-1}\cdots A_{3,4}^{-1}\cdots A_{3,4}\cdots A_{3,4}^{-1}$,
and so on. For $A\in C_{3}\wr S_{N}$ uniformly random and every $q\in\mathbb{Z}$,
conditioning on that $A_{3,4}$ is non-zero, the expected value of
$\left[A_{3,4}\right]^{q}$ is $\mathbf{1}_{q\equiv0\pmod3}$. Fortunately,
this property is a feature of the core graphs and not only of the
particular assignment: by definition, $H\in\XF{\left\langle w\right\rangle }$
if and only if $w\in H$ and every edge of $\Gamma_{X}\left(H\right)$
is covered by some edge of $\Gamma_{X}\left(\left\langle w\right\rangle \right)$
in the graph morphism $\Gamma_{X}\left(\left\langle w\right\rangle \right)\to\Gamma_{x}\left(H\right)$.
In other words, the closed path at the root of $\Gamma_{X}\left(H\right)$
which corresponds to $w$ must go through every edge of the graph.
The restriction that every non-zero entry repeats some multiple of
$3$ times in the monomial (counted with signs), is equivalent to
that the path of $w$ goes through every edge a total signed number
of times which is a multiple of $3$. This generalizes to the following
explicit form of Lemma \ref{lem:rational function}:
\begin{lem}
\label{lem:explicit rational expression}Let $w\in\F_{r}$. For every
$m\in\mathbb{Z}_{\ge2}$ denote
\[
{\cal Q}_{m}\left(w\right)=\left\{ H\in\XF{\left\langle w\right\rangle }\,\middle|\,\begin{gathered}\mathrm{the~number~of~times~}w~\mathrm{traverses~every~edge~of~\Gamma_{X}\left(H\right)},\\
\mathrm{counted~with~signs,~is~a~multiple~of~}m
\end{gathered}
\right\} .
\]
Then for every large enough $N$,
\begin{equation}
\trw\left(\cmsn\right)=\sum_{H\in{\cal Q}_{m}\left(w\right)}L_{H,\F_{r}}^{X}\left(N\right).\label{eq:rational formula for cmsn}
\end{equation}
Likewise, denote
\[
{\cal Q}_{\infty}\left(w\right)=\left\{ H\in\XF{\left\langle w\right\rangle }\,\middle|\,\begin{gathered}w~\mathrm{traverses~every~edge~of~\Gamma_{X}\left(H\right)}\\
\mathrm{the~same~number~of~times~in~each~direction}
\end{gathered}
\right\} .
\]
Then for every large enough $N$,
\begin{equation}
\trw\left(\ssn\right)=\sum_{H\in{\cal Q}_{\infty}\left(w\right)}L_{H,\F_{r}}^{X}\left(N\right).\label{eq:rational formula for s1sn}
\end{equation}
\end{lem}

Although this is not explicit from the notation, note that because
the set $\XF{\left\langle w\right\rangle }$ depends on the choice
of basis $X$, so do the sets ${\cal Q}_{m}\left(w\right)$ and ${\cal Q}_{\infty}\left(w\right)$.

As an example, if $w=x_{1}^{2}x_{2}^{2}$, six of the seven subgroups
in $\XF{\left\langle w\right\rangle }$ belong to none of ${\cal Q}_{m}\left(w\right)$
($m\in\mathbb{Z}_{\ge2}\cup\left\{ \infty\right\} $). The remaining
subgroup is $\F_{2}$ itself, with core graph \foreignlanguage{english}{}\selectlanguage{english}%
$\vphantom{\Big|}\Gamma=\xymatrix@1{\otimes\ar@(dl,ul)[]^{x_{1}}\ar@(dr,ur)[]_{x_{2}}}$\selectlanguage{american}%
\foreignlanguage{english}{, which belongs to ${\cal Q}_{2}\left(w\right)$
but does not belong to ${\cal Q}_{m}\left(w\right)$ for $m\in\mathbb{Z}_{\ge3}\cup\left\{ \infty\right\} $.
Thus, $\tr_{x_{1}^{2}x_{2}^{2}}\left(\ctsn\right)=\frac{1}{N}$ for
every $N\ge1$, whereas $\tr_{x_{1}^{2}x_{2}^{2}}\left(\cmsn\right)=\tr_{x_{1}^{2}x_{2}^{2}}\left(\ssn\right)=0$
for every $m\in\mathbb{Z}_{\ge3}$ and $N\ge1$. (Note how this agrees
with Theorem \ref{thm:Frobenius}.)}

To say how large $N$ should be, one needs to go over the elements
of ${\cal Q}_{m}\left(w\right)$. However, as every edge in $\Gamma_{X}\left(H\right)$
is covered at least twice (in the same direction or in different directions),
the formulas in Lemma \ref{lem:explicit rational expression} holds
at least for $N\ge\frac{1}{2}\max_{i\in\left[r\right]}\#E_{i}\left(w\right)$.\medskip{}

Let $H\le\F_{r}$ be a subgroup containing $w$. Our next observation
is that the conditions above regarding how many times $w$ traverses
every edge of $\Gamma_{X}\left(H\right)$ are, in fact, algebraic:
\begin{lem}
Let $w\in\F_{r}$ and $H\in\F_{r}$ be a subgroup containing $w$,
and let $m\in\mathbb{Z}_{\ge2}\cup\left\{ \infty\right\} $. The number
of times, counted with signs, that $w$ traverses every edge in $\Gamma_{X}\left(H\right)$
is a multiple of $m$ (or $0$ if $m=\infty$), if and only if $w\in K_{m}\left(H\right)$.
\end{lem}

\begin{proof}
Recall that $w\in K_{m}\left(H\right)$ if and only if when $w$ is
written as a word in a fixed but arbitrary basis, the total exponent
of every generator, counted with signs, is zero modulo $m$ (or zero
if $m=\infty$). Let $T$ be any spanning tree in the core graph $\Gamma_{X}\left(H\right)$.
There are $\mathrm{rank}\left(H\right)$ edges outside the tree, and
after an arbitrary orientation of these $\mathrm{rank}\left(H\right)$
edges, we obtain a basis for $H$: the element associated with the
oriented edge $\overrightarrow{e}$ is the one corresponding to the
closed path which goes from the root of $\Gamma_{X}\left(H\right)$
to the origin of $\overrightarrow{e}$ through $T$, traverses $\overrightarrow{e}$,
and returns to the root through $T$. Recall that every element of
$H$ corresponds to a closed, non-backtracking path at the base-point
of $\Gamma_{X}\left(H\right)$, and to write this element in the basis
we have just constructed, we simply keep track of each time the corresponding
closed path traverses one of the edges outside the spanning tree.
Now, if $w\in H$ traverses every edge of $\Gamma_{X}\left(H\right)$
a multiple of $m$ times, then any choice of spanning tree shows that
$w\in K_{m}\left(H\right)$.

Conversely, assume that $w\in K_{m}\left(H\right)$, and let $e$
be an edge of $\Gamma_{X}\left(H\right)$. If $e$ is not a bridge
(namely, not a separating edge the removal of which disconnects the
graph), then there is a spanning tree not containing $e$ and thus
$w$ traverses $e$ a multiple-of-$m$ times. If $e$ is a bridge,
then every closed path traverses it in a ``balanced'' fashion, namely,
the same number of times in each of the two directions.
\end{proof}
\begin{cor}
\label{cor:Q_m(w) defined algebraically}For $m\in\mathbb{Z}_{\ge2}\cup\left\{ \infty\right\} $,
\[
{\cal Q}_{m}\left(w\right)=\left\{ H\in\XF{\left\langle w\right\rangle }\,\middle|\,w\in K_{m}\left(H\right)\right\} .
\]
\end{cor}

The subgroups of minimal rank in ${\cal Q}_{m}\left(w\right)$ coincide
with the subgroups of minimal rank among those containing $w$ in
their ``$m$-kernel'':
\begin{lem}
\label{lem:minimal rank in Q_m(w)}For $w\in\F_{r}$ and $m\in\mathbb{Z}_{\ge2}\cup\left\{ \infty\right\} $,
the subgroups of minimal rank in ${\cal Q}_{m}\left(w\right)$ are
precisely 
\begin{equation}
\left\{ H\le\F_{r}\,\middle|\,w\in K_{m}\left(H\right),~\mathrm{rank}\left(H\right)=1-\ecm\left(w\right)\right\} .\label{eq:critical subgroups}
\end{equation}
In particular, the number of subgroup in the set \eqref{eq:critical subgroups}
is finite.
\end{lem}

\begin{proof}
First, by Corollary \ref{cor:Q_m(w) defined algebraically}, every
group $H\in{\cal Q}_{m}\left(w\right)$ satisfies $w\in K_{m}\left(H\right)$,
and so every subgroup in ${\cal Q}_{m}\left(w\right)$ has rank at
least $1-\ecm\left(w\right)$. if $\ecm\left(w\right)=-\infty$, then
both sets considered in the lemma are empty.

Assume $\ecm\left(w\right)>-\infty$, and let $H\le\F_{r}$ satisfy
$w\in K_{m}\left(H\right)$ and $1-\mathrm{rank}\left(H\right)=\ecm\left(w\right)$.
We claim that $H\in{\cal Q}_{m}\left(w\right)$ -- this would show
that the minimal rank of subgroups in ${\cal Q}_{m}\left(w\right)$
is precisely $1-\ecm\left(w\right)$ and that the two sets considered
are identical. Indeed, assume by contradiction that $H\notin{\cal Q}_{m}\left(w\right)$.
Then the morphism $\Gamma_{X}\left(\left\langle w\right\rangle \right)\to\Gamma_{X}\left(H\right)$
is \emph{not} surjective, and the image of $\Gamma_{X}\left(\left\langle w\right\rangle \right)$
in $\Gamma_{X}\left(H\right)$ is a subgraph which is the core graph
of some $M\in{\cal Q}_{m}\left(w\right)$, and in particular $w\in K_{m}\left(M\right)$.
Moreover, $M$ is then a proper free factor of $H$ and thus has smaller
rank than $H$. This is impossible as it contradicts the definition
of $\ecm\left(w\right)$. Thus $H\in{\cal Q}_{m}\left(w\right)$.

Finally, as $\XF{\left\langle w\right\rangle }$ is finite, so is
${\cal Q}_{m}\left(w\right)$ and therefore so is the set in \eqref{eq:critical subgroups}.
\end{proof}
We can now complete the proof of Theorem \ref{thm:trace-of-generalized-permutations - baby version}
and show that for $G\left(N\right)=\ssn$ (when $m=\infty$) or $G\left(N\right)=$$\trw\left(\cmsn\right)$
(for $m\in\mathbb{Z}_{\ge2}$), $\trw\left(G\left(N\right)\right)=C\cdot N^{\chi_{m}\left(w\right)}+O\left(N^{\chi_{m}\left(w\right)-1}\right)$
with $C$ the size of the set in \eqref{eq:critical subgroups}.
\begin{proof}[Proof of Theorem \ref{thm:trace-of-generalized-permutations - baby version}]
 Let $m\in\mathbb{Z}_{\ge2}\cup\left\{ \infty\right\} $ and $G\left(N\right)=\cmsn$
or $G\left(N\right)=\ssn$ accordingly. The summand corresponding
to $H\in{\cal Q}_{m}\left(w\right)$ in \eqref{eq:rational formula for cmsn}
or in \eqref{eq:rational formula for s1sn} has leading term $N^{\#V\left(H\right)-\#E\left(H\right)}=N^{1-\mathrm{rank}\left(H\right)}$,
and so the summand is $N^{1-\mathrm{rank}\left(H\right)}+O\left(N^{-\mathrm{rank}\left(H\right)}\right)$.
By Lemma \ref{lem:minimal rank in Q_m(w)}, there are precisely $C$
elements in ${\cal Q}_{m}\left(w\right)$ of the minimal rank $1-\ecm\left(w\right)$
and all others have larger rank. Therefore

\[
\trw\left(G\left(N\right)\right)=C\cdot N^{\ecm\left(w\right)}+O\left(N^{\ecm\left(w\right)-1}\right).
\]
\end{proof}

\subsection{\label{subsec:The-second-term}The second term of the rational expressions}

Lemma \ref{lem:rational function} shows that for generalized symmetric
groups $G\left(N\right)$, the expected trace $\trw\left(G\left(N\right)\right)$
is given by a rational expression in $N$, and Theorem \ref{thm:trace-of-generalized-permutations - baby version}
gives an algebraic interpretation for the leading term of this expression.
We now want to strengthen Theorem \ref{thm:trace-of-generalized-permutations - baby version}
and show that the rational expression does not only tell us about
the subgroup $H\le\F_{r}$ of minimal rank with the property that
$w\in K_{m}\left(H\right)$, but also about the ``second'' minimal
group. If there is more than one group of minimal rank, this is already
captured by Theorem \ref{thm:trace-of-generalized-permutations - baby version}.
But we want to deal also with the case that there is a unique subgroup
as above of minimal rank.

To define the second minimal group, we do not rely on the set ${\cal Q}_{m}\left(w\right)$
which depends on the given basis $X$. Instead, we consider only \emph{algebraic
extensions} of $\left\langle w\right\rangle $ which also contain
$w$ in their ``$m$-kernel''. Namely, for $w\in\F_{r}$ and $m\in\mathbb{Z}_{\ge2}\cup\left\{ \infty\right\} $,
denote\marginpar{$\protect\ae_{m}\left(w\right)$} 
\[
\ae_{m}\left(w\right)\stackrel{\mathrm{def}}{=}\left\{ A\le\F_{r}\,\middle|\,\left\langle w\right\rangle \alg A~\mathrm{and}~w\in K_{m}\left(A\right)\right\} .
\]
If $A$ is an algebraic extension of $\left\langle w\right\rangle $
then $\left\langle w\right\rangle $ $X$-covers $A$ for every basis
$X$. Indeed, if $w\in A$ and $\left\langle w\right\rangle $ does
not $X$-cover $A$ then the image of $\Gamma_{X}\left(\left\langle w\right\rangle \right)$
in $\Gamma_{X}\left(A\right)$ constitutes an intermediate subgroup
which is a proper free factor of $A$. In particular, $\ae_{m}\left(w\right)\subseteq{\cal Q}_{m}\left(w\right)$.
Moreover, all subgroups $H$ of minimal rank with $w\in K_{m}\left(H\right)$
are algebraic extensions of $\left\langle w\right\rangle $, because
if $w\in K_{m}\left(H\right)$ and $w\in A\ff H$, then clearly $w\in K_{m}\left(A\right)$,
so $H$ cannot be of minimal rank unless it is an algebraic extension.
Thus
\[
\ecm\left(w\right)=1-\min_{A\in\ae_{m}\left(w\right)}\mathrm{rank}\left(A\right).
\]

\begin{defn}
Let $w\in\F_{r}$ and $m\in\mathbb{Z}_{\ge2}\cup\left\{ \infty\right\} $.
If $\left|\ae_{m}\left(w\right)\right|\le1$, define $\ecm^{\left(2\right)}\left(w\right)\stackrel{\mathrm{def}}{=}-\infty$
and $C_{m}^{\left(2\right)}\stackrel{\mathrm{def}}{=}0$. Otherwise,
let $A\in\ae_{m}\left(w\right)$ be an arbitrary subgroup of minimal
rank, and define
\[
\ecm^{\left(2\right)}\left(w\right)\stackrel{\mathrm{def}}{=}1-\min_{B\in\ae_{m}\left(w\right)\setminus\left\{ A\right\} }\mathrm{rank}\left(B\right).
\]
Also, define $C_{m}^{\left(2\right)}\left(w\right)$ to be the number
of subgroups $B$ in $\ae_{m}\left(w\right)\setminus\left\{ A\right\} $
of minimal rank, namely, with $\mathrm{rank}\left(B\right)=1-\ecm^{\left(2\right)}\left(w\right)$.
\end{defn}

Note that the numbers $\chi_{m}^{\left(2\right)}\left(w\right)$ and
$C_{m}^{\left(2\right)}\left(w\right)$ do not depend on the arbitrary
subgroup $A$. If the constant $C$ from Theorem \ref{thm:trace-of-generalized-permutations - baby version}
is at least two, then $\ecm^{\left(2\right)}\left(w\right)=\ecm\left(w\right)$
and $C_{m}^{\left(2\right)}\left(w\right)=C-1$. If $C=1$, then $\ecm^{\left(2\right)}\left(w\right)<\ecm\left(w\right)$.
\begin{thm}
\label{thm:two min rank}Fix $w\in\F_{r}$ and let $m\in\mathbb{Z}_{\ge2}$
in which case $G\left(N\right)=\cmsn$, or $m=\infty$ in which case
$G\left(N\right)=\ssn$. Then
\[
\trw\left(G\left(N\right)\right)=N^{\ecm\left(w\right)}+C_{m}^{\left(2\right)}\left(w\right)\cdot N^{\ecm^{\left(2\right)}\left(w\right)}+O\left(N^{\ecm^{\left(2\right)}\left(w\right)-1}\right).
\]
\end{thm}

The point of Theorem \ref{thm:two min rank} is that one can always
read off from the expression for $\trw\left(G\left(N\right)\right)$
the ranks of the two subgroups of minimal rank in $\ae_{m}\left(w\right)$.
In particular, we get the following corollary which we use below in
the proof of Theorem \ref{thm:surface words are separable}:
\begin{cor}
\label{cor:trw=00003Dn^chi}Fix $w\in\F_{r}$ and let $m\in\mathbb{Z}_{\ge2}$
in which case $G\left(N\right)=\cmsn$, or $m=\infty$ in which case
$G\left(N\right)=\ssn$. Then $\trw\left(G\left(N\right)\right)$
is of the form $N^{\chi}$ (for some $\chi\in\mathbb{Z}$) if and
only if $\left|\ae_{m}\left(w\right)\right|=1$.
\end{cor}

\begin{proof}[Proof of Theorem \ref{thm:two min rank}]
 We claim that
\[
{\cal Q}_{m}\left(w\right)=\bigcup_{A\in\ae_{m}\left(w\right)}\left\{ H\le\F_{r}\,\middle|\,A\ff_{\Xcov}H\right\} .
\]
Indeed, as mentioned in Section \ref{sec:Prelimiaries} above, for
every extension of free groups $J_{1}\le J_{2}$, there is a unique
intermediate subgroup $A$ such that $J_{1}\alg A\ff J_{2}$ (see
\cite[Claim 4.5]{PP15} for the proof in the finitely generated case,
which is the case we need here). If $H\in{\cal Q}_{m}\left(w\right)$
and $A$ is the unique intermediate subgroup $\left\langle w\right\rangle \alg A\ff H$
then $w\in K_{m}\left(A\right)$ and so $A\in\ae_{m}\left(w\right)$.
Moreover, in this case $A\covers H$ because the surjective core-graph
morphism $\Gamma_{X}\left(\left\langle w\right\rangle \right)\to\Gamma_{X}\left(H\right)$
factors as $\Gamma_{X}\left(\left\langle w\right\rangle \right)\stackrel{\alpha_{1}}{\to}\Gamma_{X}\left(A\right)\stackrel{\alpha_{2}}{\to}\Gamma_{X}\left(H\right)$,
so $\alpha_{2}$ must too be surjective. On the other hand, if $A\in\ae_{m}\left(w\right)$
and $A\ff_{\Xcov}H$ then $\left\langle w\right\rangle \covers A\covers H$
and as ``$\covers$'' is transitive, $\left\langle w\right\rangle \covers H$.
In addition, $w\in K_{m}\left(A\right)\le K_{m}\left(H\right)$. Hence
$H\in{\cal Q}_{m}\left(w\right)$.

Thus, we get from Lemma \ref{lem:explicit rational expression} that
if for an arbitrary finitely generated subgroup $A\le\F_{r}$ we denote
\begin{equation}
\mathrm{contrib}_{A}\left(N\right)\stackrel{\mathrm{def}}{=}\sum_{H\le\F_{r}\thinspace\mathrm{s.t.}~A\ff_{\Xcov}H}L_{H,\F_{r}}^{X}\left(N\right),\label{eq:contrib}
\end{equation}
then
\begin{equation}
\trw\left(G\left(N\right)\right)=\sum_{A\in\ae_{m}\left(w\right)}\mathrm{contrib}_{A}\left(N\right).\label{eq:sum by AE_m}
\end{equation}
Now let $A\le\F_{r}$ be an arbitrary finitely generated subgroup.
Because $L_{H,\F_{r}}^{X}\left(N\right)=N^{1-\mathrm{rank}\left(H\right)}+O\left(N^{-\mathrm{rank}\left(H\right)}\right)$,
and all the subgroups $H$ in \eqref{eq:contrib} satisfy $\mathrm{rank}\left(A\right)\le\mathrm{rank}\left(H\right)$
with equality if and only if $A=H$, it is clear that
\begin{equation}
\mathrm{contrib}_{A}\left(N\right)=N^{1-\mathrm{rank}A}+O\left(N^{-\mathrm{rank}A}\right).\label{eq:first upper bound on contrib}
\end{equation}
Recall Theorem \ref{thm:PP15 subgroups} above (originally \cite[Theorem 1.8]{PP15}),
by which the expected number of points in $\left\{ 1,\ldots,N\right\} $
fixed by all elements of $\varphi\left(A\right)$ in a uniformly random
$\varphi\in\mathrm{Hom}\left(\F_{r},S_{N}\right)$ is
\begin{equation}
N^{1-\mathrm{rank}\left(A\right)}+C\cdot N^{1-\pi\left(A\right)}+O\left(N^{-\pi\left(A\right)}\right),\label{eq:thm 1.8 for L}
\end{equation}
where $C$ is the number of subgroups $J\le\F_{r}$ of rank $\pi\left(A\right)$
and which contain $A$ but not as a free factor. Parallel to \eqref{eq:formula for fixed points of a word in S_N},
which gives a formula for the expected number of fixed points of a
single word, the expected number of common fixed points of $A$ can
be computed by
\begin{equation}
\sum_{H\in\XF A}L_{H,\F_{r}}^{X}\left(N\right)\label{eq:formula for fixed points of subgroup}
\end{equation}
(recall that $\XF A=\left\{ H\le\F_{r}\,\middle|\,A\covers X\right\} $)
-- see \cite[Section 5]{Puder2014}. As in the case of single words,
$\pi\left(A\right)$ is precisely the smallest rank of a proper algebraic
extension of $A$. The primitivity rank of $A$ is sometimes smaller
than or equal to $\mathrm{rank}\left(A\right)$, but in the cases
where $\pi\left(A\right)>\mathrm{rank}\left(A\right)$, Theorem \ref{thm:PP15 subgroups}
and \eqref{eq:thm 1.8 for L} can be interpreted as follows: in the
formula \eqref{eq:formula for fixed points of subgroup} there is
a leading term of $N^{1-\mathrm{rank}\left(A\right)}$ coming from
$L_{A,\F_{r}}^{X}\left(N\right)$, but then all contributions coming
from free extensions of $A$ in $\XF A$, together with $L_{A,\F_{r}}^{X}\left(N\right)-N^{1-\mathrm{rank}\left(A\right)}$,
cancel out in all terms of order $N^{-\mathrm{rank}\left(A\right)},N^{-\mathrm{rank}\left(A\right)-1},\ldots,N^{1-\pi\left(A\right)}$.
(The positive coefficient of $N^{1-\pi\left(A\right)}$ comes from
algebraic extensions of $A$, not from free extensions.) Hence,
\begin{equation}
\mathrm{contrib}_{A}\left(N\right)=N^{1-\mathrm{rank}\left(A\right)}+O\left(N^{-\pi\left(A\right)}\right).\label{eq:second upper bound on contrib}
\end{equation}

We may assume that $\ecm^{\left(2\right)}\left(w\right)<\ecm\left(w\right)$,
for otherwise Theorem \ref{thm:two min rank} follows immediately
from Theorem \ref{thm:trace-of-generalized-permutations - baby version}.
Let $A_{0}\in\ae_{m}\left(w\right)$ be the unique algebraic extension
with $w\in K_{m}\left(A_{0}\right)$ of rank $1-\ecm\left(w\right)$,
and let $A_{1},\ldots,A_{C_{m}^{\left(2\right)}\left(w\right)}$ be
those of rank $1-\ecm^{\left(2\right)}\left(w\right)$. Because algebraic
extensions is a transitive relation and $H\le J\Longmapsto K_{m}\left(H\right)\le K_{m}\left(J\right)$,
every proper algebraic extension of $A_{0}$ is also in $\ae_{m}\left(w\right)$
and so its rank is at least $1-\ecm^{\left(2\right)}\left(w\right)$.
In particular, $\pi\left(A_{0}\right)\ge1-\ecm^{\left(2\right)}\left(w\right)$
for every $A\in\ae_{m}\left(w\right)$. From \eqref{eq:first upper bound on contrib}
and \eqref{eq:second upper bound on contrib} it now follows that
\[
\mathrm{contrib}_{A}\left(N\right)=\begin{cases}
N^{\ecm\left(w\right)}+O\left(N^{\ecm^{\left(2\right)}\left(w\right)-1}\right) & \mathrm{if}~A=A_{0}\\
N^{\ecm^{\left(2\right)}\left(w\right)}+O\left(N^{\ecm^{\left(2\right)}\left(w\right)-1}\right) & \mathrm{if}~A=A_{1},\ldots,A_{C_{m}^{\left(2\right)}\left(w\right)}\\
O\left(N^{\ecm^{\left(2\right)}\left(w\right)-1}\right) & \mathrm{if}~A\in\ae_{m}\left(w\right)\setminus\left\{ A_{0},A_{1},\ldots,A_{C_{m}^{\left(2\right)}\left(w\right)}\right\} 
\end{cases}.
\]
Plugging these expressions in \eqref{eq:sum by AE_m} completes the
proof of Theorem \ref{thm:two min rank}.
\end{proof}
A nice corollary of Theorem \ref{thm:trace-of-generalized-permutations - baby version}
is the following. Recall from Definition \ref{def:cl} that $\cl\left(w\right)$
denotes the commutator length of $w$ and $\sql\left(w\right)$ denotes
the square length of $w$.
\begin{cor}
\label{cor:lower bound for cl,sq}We have 
\[
\chi_{\infty}\left(w\right)\ge1-2\cdot\cl\left(w\right)~~~~\mathrm{and}~~~~\chi_{2}\left(w\right)\ge1-\min\left(\sql\left(w\right),2\cl\left(w\right)\right).
\]
In particular, 
\[
\trw\left(\ssn\right)=\Omega\left(N^{1-2\cl\left(w\right)}\right)~~~~\mathrm{and}~~~~\trw\left(\ctsn\right)=\Omega\left(N^{1-\min\left(\sql\left(w\right),2\cl\left(w\right)\right)}\right).
\]
\end{cor}

\begin{proof}
If $g=\cl\left(w\right)$ then there are $2g$ words $u_{1},v_{1},\ldots,u_{g},v_{g}\in\F_{r}$
such that $w=\left[u_{1},v_{1}\right]\cdots\left[u_{g},v_{g}\right]$.
Let $J=\left\langle u_{1},v_{1},\ldots,u_{g},v_{g}\right\rangle $,
and note that $w\in K_{\infty}\left(J\right)$ and $w\in K_{2}\left(J\right)$
and that $\mathrm{rank}\left(J\right)\le2g$. Hence $\chi_{\infty}\left(w\right)\ge1-\rk\left(J\right)\ge1-2g$
and likewise $\ect\left(w\right)\ge1-2g$.

Similarly, if $g=\sql\left(w\right)$ then there are $g$ words $u_{1},\ldots,u_{g}\in\F_{r}$
such that $w=u_{1}^{2}\cdots u_{g}^{2}$. Let $J=\left\langle u_{1},\ldots,u_{g}\right\rangle $,
and note that $w\in K_{2}\left(J\right)$ and that $\mathrm{rank}\left(J\right)\le g$.
Thus $\chi_{2}\left(w\right)\ge1-\rk\left(J\right)\ge1-g$.
\end{proof}

\section{\label{sec:Surface-Words-proof}Surface words and the proof of Theorem
\ref{thm:surface words are separable}}

\subsection{Orientable surface words}

First, we prove the first part of Theorem \ref{thm:surface words are separable},
which deals with the orientable surface word $\left[x_{1},y_{1}\right]\cdots\left[x_{g},y_{g}\right]$.
\begin{proof}[Proof of Theorem \ref{thm:surface words are separable}, orientable
case]
 Assume that some word $w\in\F_{r}$ induces the same measure as\linebreak{}
$\left[x_{1},y_{1}\right]\cdots\left[x_{g},y_{g}\right]$ on every
compact group $G$. In particular, the expected value of any irreducible
character $\psi$ of $G$ under the $w$-measure is $\left(\psi\left(e\right)\right)^{1-2g}$.
In the case of the unitary groups $\U\left(N\right)$, the trace is
an irreducible $N$-dimensional character and thus $\trw\left(\U\left(N\right)\right)=N^{1-2g}$.
From Theorem \ref{thm:Un} it now follows that $\cl\left(w\right)\le g$.
In particular, $w\in\left[\F_{r},\F_{r}\right]$.

On the other hand, the trace of $\ssn$ is also an $N$-dimensional
irreducible character, and so $\trw\left(\ssn\right)=N^{1-2g}$. From
Theorem \ref{thm:trace-of-generalized-permutations - baby version}
it follows that $\eci\left(w\right)=1-2g$. If $w=\left[u_{1},v_{1}\right]\cdots\left[u_{\cl\left(w\right)},v_{\cl\left(w\right)}\right]$
then, as in Corollary \ref{cor:lower bound for cl,sq} and its proof,
\begin{equation}
1-2g=\eci\left(w\right)\ge1-\mathrm{rank}\left(J\right)\ge1-2\cl\left(w\right),\label{eq:g and cl}
\end{equation}
where $J=\left\langle u_{1},v_{1},\ldots,u_{\cl\left(w\right)},v_{\cl\left(w\right)}\right\rangle $,
and we obtain that $\cl\left(w\right)\ge g$.

Thus $\cl\left(w\right)=g$. Moreover, all the weak inequalities in
\eqref{eq:g and cl} are equalities, and $\mathrm{rank}\left(J\right)=2\cl\left(w\right)=1-\eci\left(w\right)$.
This shows that $J$ has minimal rank among the subgroups with $w\in K_{\infty}\left(w\right)$,
and so $J\in\ae_{\infty}\left(w\right)$. In addition, $J$ is a free
factor of $\F_{r}$: otherwise, it would have a non-trivial algebraic
extension $J\lvertneqq_{\mathrm{alg}}A\ff\F_{r}$, and then $A\in\ae_{\infty}\left(w\right)$
and $\left|\ae_{\infty}\left(w\right)\right|\ge2$, in contradiction
to Corollary \ref{cor:trw=00003Dn^chi} which applies in this case.

As $\mathrm{rank}\left(J\right)=2\cl\left(w\right)=2g$, the words
$u_{1},v_{1},\ldots,u_{g},v_{g}$ are free and constitute a basis
for $J$, and as $H\ff\F_{r}$, they are part of a basis for $\F_{r}$.
Therefore $r\ge2g$ and $w\auteqr\left[x_{1},y_{1}\right]\cdots\left[x_{g},y_{g}\right]$.
\end{proof}
\begin{rem}
\label{rem:s1sn can be replaced by finite groups}We mentioned above
that Conjecture \ref{conj:shalev minus} sometimes appears in the
literature in a stronger version, where only finite groups are involved
rather than all compact groups. In our proof of the conjecture for
the case of orientable surface words, we used two compact infinite
groups: $\U\left(N\right)$ and $\ssn$. However, the latter can be
easily replaced by finite groups: let $\left|w\right|$ denote the
length of the word $w$. If $m>\left|w\right|$, then $w$ cannot
traverse any edge of a core graph $m$ times, $2m$ times, or $-m$
times (when counted with signs). So in this case, if $w\in K_{m}\left(H\right)$
then also $w\in K_{\infty}\left(H\right)$, and $\trw\left(\cmsn\right)=\trw\left(\ssn\right)$
for all $N$. So for every $w$, one may replace the group $\ssn$
in the proof above with the group $\cmsn$ for any $m\ge\left|w\right|$.
This means that the only infinite groups one actually needs for the
proof are $\U\left(N\right)$. See also Question \ref{enu:remove compact}
in Section \ref{sec:Open-Questions}.
\end{rem}

\subsection{Non-orientable surface words}
\begin{proof}[Proof of Theorem \ref{thm:surface words are separable}, non-orientable
case]
 Assume that some word $w\in\F_{r}$ induces the same measure as
$x_{1}^{2}\cdots x_{g}^{2}$ on every compact group. In particular,
the expected value of any irreducible character $\psi$ of $G$ under
the $w$-measure in $\frac{\left(\fs_{\psi}\right)^{g}}{\left(\dim\psi\right)^{g-1}}$.
In the case of the group $\ssn$, the trace is an irreducible $N$-dimensional
character with $\fs=0$ and thus $\trw\left(\ssn\right)=0$. From
Corollary \ref{cor:lower bound for cl,sq} we deduce that $\cl\left(w\right)=\infty$,
namely, $w\notin\left[\F_{r},\F_{r}\right]$.

In the case of the orthogonal groups $\O\left(N\right)$, the trace
is an irreducible $N$-dimensional real character with $\fs=1$ and
thus $\trw\left(\O\left(N\right)\right)=N^{1-g}$. As $\cl\left(w\right)=\infty$,
Theorem \ref{thm:On} says in this case that $\trw\left(\O\left(N\right)\right)=O\left(N^{1-\sql\left(w\right)}\right)$.
It follows that $\sql\left(w\right)\le g$.

On the other hand, the trace of $\ctsn$ is also $N$-dimensional
irreducible with $\fs=1$, and so $\trw\left(\ctsn\right)=N^{1-g}$.
From Theorem \ref{thm:trace-of-generalized-permutations - baby version}
it follows that $\ect\left(w\right)=1-g$. If $w=u_{1}^{2}\cdots u_{\sql\left(w\right)}^{2}$
then $w\in K_{2}\left(J\right)$ where $J=\left\langle u_{1},\ldots,u_{\sql\left(w\right)}\right\rangle $.
Hence 
\begin{equation}
1-g=\ect\left(w\right)\ge1-\mathrm{rank}\left(J\right)\ge1-\sql\left(w\right)\label{eq:g and sql}
\end{equation}
and we obtain that $\sql\left(w\right)\ge g$.

Thus $\sql\left(w\right)=g$. Moreover, all the weak inequalities
in \eqref{eq:g and sql} are equalities, and $\mathrm{rank}\left(J\right)=\sql\left(w\right)=1-\ect\left(w\right)$.
This shows that $J$ has minimal rank among the subgroups with $w\in K_{2}\left(w\right)$,
and so $J\in\ae_{2}\left(w\right)$. In addition, $J$ is a free factor
of $\F_{r}$: otherwise, it would have a non-trivial algebraic extension
$J\lvertneqq_{\mathrm{alg}}A\ff\F_{r}$, and then $A\in\ae_{2}\left(w\right)$
and $\left|\ae_{2}\left(w\right)\right|\ge2$, in contradiction to
Corollary \ref{cor:trw=00003Dn^chi} which applies in this case.

As $\mathrm{rank}\left(J\right)=\sql\left(w\right)=g$, the words
$u_{1},\ldots,u_{g}$ are free and constitute a basis for $J$, and
as $J\ff\F_{r}$, they are part of a basis for $\F_{r}$. Therefore
$r\ge g$ and $w\auteqr x_{1}^{2}\cdots x_{g}^{2}$.
\end{proof}
\begin{rem}
\label{rem:s1sn can be replace by finite - again}As in the orientable
case, the small role of the infinite group $\ssn$ in the last proof
can be also played by the groups $\cmsn$ for large enough $m$.
\end{rem}

\section{\label{sec:Open-Questions}Open Questions}

We conclude with some open questions naturally arising from the results
in this paper.
\begin{enumerate}
\item \label{enu:remove compact}Can Theorem \ref{thm:surface words are separable}
be proven also based on word measures on finite groups only? Namely,
can the role played in the proof by $\U\left(N\right)$ and $\O\left(N\right)$
be also played by some finite groups? (And see Remarks \ref{rem:s1sn can be replaced by finite groups}
and \ref{rem:s1sn can be replace by finite - again}.)
\item \label{enu:words with single alg ext}Corollary \ref{cor:trw=00003Dn^chi}
has the potential of yielding a solution of more special cases of
Conjecture \ref{conj:shalev minus}. Namely, if there is a relatively
small set of $\mathrm{Aut}\big(\F_{r}\big)$-orbits, along surface
words, with the property that $\left|\ae_{m}\right|=1$ for some $m\in\mathbb{Z}_{\ge2}\cup\left\{ \infty\right\} $,
then one can hope to prove Conjecture \ref{conj:shalev minus} for
these orbits. Let us mention two examples: the words $\left[x,y\right]^{2}$
and $\left[x,y\right]\left[x,z\right]$ both satisfy that $\left|\ae_{\infty}\left(w\right)\right|=1$.
\end{enumerate}

\section*{Acknowledgments}

We thank Henry Wilton and Liviu Pãunescu for beneficial comments.
D.P.~was supported by ISF grant 1071/16. 

\selectlanguage{english}%
\bibliographystyle{alpha}
\bibliography{surface-words}

\begin{thebibliography}{MVW07}

\bibitem[AV11]{Amit2011}
Alon Amit and Uzi Vishne.
\newblock Characters and solutions to equations in finite groups.
\newblock {\em J. Algebra Appl.}, 10(4):675--686, 2011.

\bibitem[Fro96]{frobenius1896gruppencharaktere}
Georg Frobenius.
\newblock {\"U}ber gruppencharaktere.
\newblock {\em Sitzungsberichte Akademie der Wissenschaften zu Berlin}, pages
  985--1021, 1896.

\bibitem[FS06]{frobenius1906reellen}
Georg Frobenius and Issai Schur.
\newblock {\"U}ber die reellen darstellungen der endlichen gruppen.
\newblock {\em Sitzungsberichte Akademie der Wissenschaften zu Berlin}, pages
  186--208, 1906.

\bibitem[HMP19]{Hanani}
Liam Hanani, Chen Meiri, and Doron Puder.
\newblock Some orbits of free words that are determined by measures on finite
  groups.
\newblock In preparation, 2019.

\bibitem[Isa76]{isaacs1994character}
I.~Martin Isaacs.
\newblock {\em Character theory of finite groups}, volume~69.
\newblock Academic Press, 1976.

\bibitem[LP10]{Linial2010}
Nati Linial and Doron Puder.
\newblock Word maps and spectra of random graph lifts.
\newblock {\em Random Structures and Algorithms}, 37(1):100--135, 2010.

\bibitem[MP15]{MP}
Michael Magee and Doron Puder.
\newblock Word measures on unitary groups.
\newblock arXiv preprint 1509.07374, 2015.

\bibitem[MP19a]{MP-Un}
Michael Magee and Doron Puder.
\newblock Matrix group integrals, surfaces, and mapping class groups {I}:
  ${U}(n)$.
\newblock {\em Inventiones Mathematicae}, 2019.
\newblock to appear, available at arXiv:1802.04862 v2.

\bibitem[MP19b]{MP-On}
Michael Magee and Doron Puder.
\newblock Matrix group integrals, surfaces, and mapping class groups {II}:
  ${O}(n)$ and ${S}p(n)$.
\newblock preprint arXiv:1904.13106, 2019.

\bibitem[MVW07]{miasnikov2007algebraic}
Alexei Miasnikov, Enric Ventura, and Pascal Weil.
\newblock Algebraic extensions in free groups.
\newblock In {\em Geometric group theory}, pages 225--253. Springer, 2007.

\bibitem[Nic94]{nica1994number}
Alexandru Nica.
\newblock On the number of cycles of given length of a free word in several
  random permutations.
\newblock {\em Random Structures \& Algorithms}, 5(5):703--730, 1994.

\bibitem[PP15]{PP15}
Doron Puder and Ori Parzanchevski.
\newblock Measure preserving words are primitive.
\newblock {\em Journal of the American Mathematical Society}, 28(1):63--97,
  2015.

\bibitem[PS14]{parzanchevski2014fourier}
Ori Parzanchevski and Gili Schul.
\newblock On the {F}ourier expansion of word maps.
\newblock {\em Bulletin of the London Mathematical Society}, 46(1):91--102,
  2014.

\bibitem[Pud14]{Puder2014}
Doron Puder.
\newblock Primitive words, free factors and measure preservation.
\newblock {\em Israel J. Math.}, 201(1):25--73, 2014.

\bibitem[PW19]{West}
Doron Puder and Danielle West.
\newblock Word measures on ${GL}_n({F}_q)$.
\newblock In preparation, 2019.

\bibitem[Sha13]{Shalev2013}
Aner Shalev.
\newblock Some results and problems in the theory of word maps.
\newblock In L.~Lov{\'a}sz, I.~Ruzsa, V.T. S{\'o}s, and D.~Palvolgyi, editors,
  {\em Erd\"{o}s Centennial (Bolyai Society Mathematical Studies)}, pages
  611--650. Springer, 2013.

\bibitem[Sta83]{stallings1983topology}
John~R. Stallings.
\newblock Topology of finite graphs.
\newblock {\em Inventiones Mathematicae}, 71(3):551--565, 1983.

\bibitem[Sti12]{stillwell2012classical}
John Stillwell.
\newblock {\em Classical topology and combinatorial group theory}, volume~72 of
  {\em Graduate texts in mathematics}.
\newblock Springer, 2012.

\end{thebibliography}

\noindent Michael Magee, \\
Department of Mathematical Sciences,\\
Durham University, \\
Lower Mountjoy, DH1 3LE Durham,United Kingdom\\
\texttt{michael.r.magee@durham.ac.uk}\\

\noindent Doron Puder, \\
School of Mathematical Sciences, \\
Tel Aviv University, \\
Tel Aviv, 6997801, Israel\\
\texttt{doronpuder@gmail.com}\selectlanguage{american}%

\end{document}